%% file: main.tex
\documentclass{article}
\usepackage[utf8]{inputenc}
\usepackage{geometry}
\usepackage{multirow}

\usepackage[round, authoryear, sort&compress]{natbib}
\usepackage{pgfplots}
\usepackage{pgfplotstable} % read txt csv
\usepackage{bm} % for bf in math equations
\usepackage{amsmath}
\usepackage{txfonts} % put at last
\usepackage{enumitem}
\usepackage{url}

\usepackage{dingbat} % for check mark
\usepackage{footnote}
\usepackage{xcolor}

\usepackage[ruled,linesnumbered]{algorithm2e}
\usepackage{authblk}
\usepackage[title]{appendix}

\usepackage{hyperref}
\hypersetup{
    colorlinks=true,
    linkcolor=black,
    filecolor=gray,
    urlcolor=blue,
    citecolor=blue,
}

\usepackage[most]{tcolorbox}
\definecolor{shadecolor}{rgb}{0.9,0.9,0.9}

\usepackage{booktabs}
\usepackage{tabularx} % change width

\newtheorem{lemma}{Lemma}

\newenvironment{proof}{{\noindent\it Proof.}~~}{\hfill $\square$\vskip\topsep}

\normalsize
\title{A dynamic ordering policy for a stochastic inventory problem with cash constraints}

\usepackage{xpatch}
\xpatchcmd{\author}{\relax#1\relax}{\relax\detokenize{#1}\relax}{}{}

\author[\empty]{Zhen Chen\textsuperscript{a,}\thanks{Corresponding author: robinchen@swu.edu.cn}}
\author[b]{Roberto Rossi}
\affil[a]{College of Economics and Management, Southwest University, Chongqing 400715, China}
\affil[b]{Business School, University of Edinburgh, Edinburgh EH8 9JS, United Kingdom}
\date{}

\pgfplotsset{compat=1.14}
\begin{document}

\maketitle

%  Those retailers are nanostores, which are widely present in emerging markets, or small online retailers in some Chinese e-commerce platforms.

\begin{abstract}
This paper investigates a stochastic inventory management problem in which a cash-constrained small retailer periodically purchases a product from suppliers and sells it to a market while facing non-stationary demands. In each period, the retailer's available cash restricts the maximum quantity that can be ordered. There exists a fixed ordering cost for the retailer when purchasing. We partially characterize the optimal ordering policy by showing it has an $\bf s-C$ structure: for each period, when initial inventory is above the $\bf$s threshold, no product should be ordered no matter how much initial cash it has; when initial inventory is not large enough to be a $\bf s$ threshold, it is also better to not order when initial cash is below the threshold $C$. The values of $C$ may be state-dependent and related to each period's initial inventory. A heuristic policy $(s, C(x), S)$ is proposed: when initial inventory $x$ is less than $s$ and initial cash is greater than $C(x)$, order a quantity that brings inventory as close to $S$ as possible; otherwise, do not order. We first determine the values of the controlling parameters $s$, $C(x)$ and $S$ based on the results of stochastic dynamic programming and test their performance via an extensive computational study. The results show that the $(s, C(x), S)$ policy performs well with a maximum optimality gap of less than 1\% and an average gap of approximately 0.01\%. We then develop a simple and time-efficient heuristic method for computing policy $(s, C(x), S)$ by solving a mixed-integer linear programming problem and approximate newsvendor models: the average gap for this heuristic is approximately 2\% on our test bed.

\end{abstract}


\section{Introduction}

Cash flow management is important to the survival and growth of many businesses, especially for small-to-medium-sized enterprises (SMEs), including start-up companies and small retailers. Cash shortage may disrupt a firm's smooth operation and can even lead to insolvency \citep{john2014effect}. \cite{smirat2016cash} sampled SMEs in Jordan and found that cash management practices had influence on the financial performance of those firms. A report compiled by the research firm CB Insights found that 29\% of start-ups failed because of cash crisis \citep{cbinsights2018}. For some small retailers, the maximum number of products they can purchase depends on their available cash in this period. Some of those retailers are widely present in many developing countries and called nanostores \citep{Boulaksil2018}.  On Chinese E-commerce platforms such as Taobao.com of Alibaba Group and Jd.com of JD Group, or on social software WeChat of Tecent Group, there are many small online retailers operated by only a few people or even a single individual. External financing is more difficult to obtain for those retailers than for larger business entities. Therefore, such small retailers must consider cash constraints during business operations.
%\cite{buzacott2004inventory} noted the importance of jointly considering operational and financial decisions by analyzing a cash-constrained newsvendor model.

%Nevertheless, the multi-item stochastic lot sizing problems considering cash constraints are very complex and difficult to solve at present. For ease of analysis, we investigate the cash constraints starting from the single-item level.

% Generally, many retailers sell more than one item of products. But there are some exceptional cases where the small retailer sells only one item of products. For example,  there exist coal gas retailers that sell gas to farmers who use coal gas for cooking in some Chinese rural areas. The gas retailers purchase from the gas enterprises and only sell one kind of coal gas in general. Besides, on Chinese E-commerce platform Taobao.com, which are similar to Ebay.com, it is very easy for a single individual to open an online store without many procedures. There is an upper limit for the number of selling items, but there is no lower limit. So at the early stage of online retailing for some individuals, they only sell one item. In addition, sometimes the retailer may set a strict budget for one item although they are selling several products.

An important class of inventory control problems is the single-item single-stocking location stochastic lot-sizing problem. This problem has been widely investigated since the 1950s (\cite{wagner1958dynamic}, \cite{Scarf1960}). An interesting stream in the stochastic lot-sizing literature focuses on cash constraints. Cash constraints are different from order quantity capacity constraints in that cash flow is dynamic and related to the retailer’s revenues and costs, while order quantity capacity constraints simply impose a fixed upper bound on the ordering quantity. To the best of our knowledge, existing works on multi-item stochastic inventory problems considering cash constraints either focus on single-period problems (e.g. \cite{bi2018dynamic}) or are solved by simulation in a multi-period setting due to the complexity of the problem  (e.g. \cite{Boulaksil2018}). This motivates our study on a single-item multi-period problem.

Another difference between our paper and previous works is that we consider a fixed ordering cost in the problem. In real-life transactions, fixed ordering costs do exist for some retailers. For example, small online retailers mentioned above usually purchase from their suppliers every few weeks. The transportation cost and some other expenses of each procurement can be viewed as a fixed ordering cost because they are not related to the order quantity.

%In terms of this background, we formulate a single-item stochastic lot-sizing model for the problem.

As pointed out by \cite{graves1999single}, one major theme in the continuing development of inventory theory is to incorporate more realistic assumptions about product demand into inventory models. In this work, we consider the cash-constrained single-item single-stocking location lot-sizing problem for non-stationary stochastic demands. To the best of our knowledge, the optimal policy for the non-stationary stochastic lot-sizing problem under cash constraints has not been studied before. Moreover, the computation of optimal or near-optimal policy parameters for multi-period stochastic inventory problems remains a challenging task. Our contributions to the literature are as follows.
\begin{itemize}
  \item We investigate a cash-constrained stochastic inventory system with fixed ordering cost for the retailer.

  \item We discuss the characteristics of the optimal ordering policy for the problem under non-stationary demand over a finite horizon, and discuss a partial characterisation of the optimal ordering strategy: if initial inventory is above the threshold $s$, or if the initial cash is below $C(x)$ when initial inventory $x$ is less then the threshold $s$, then there is no need for the retailer to order.

  %prove the existence of $\bf s$ and $\bf C1$ bound analytically. The existence of $\bf C2$ bound is shown by some specific numerical cases.

 \item Inspired by the former analysis, we propose an ordering policy: $(s, C(x), S)$. We first introduce an effective heuristic for determining the values of $s$, $C(x)$ and $S$ based on the results of stochastic dynamic programming. A comprehensive numerical study shows the average optimality gaps for this policy are very small and the  $(s, C(x), S)$ policy is near-optimal, with a maximum gap of less than 1\%.

  \item We finally develop an efficient algorithm to obtain the values of $s$, $C(x)$ and $S$ by mixed-integer linear programming and approximating each ordering cycle to newsvendor models. The method can solve the problem in fractions of a second, and the average gap on the test bed is approximately 2\%.
\end{itemize}

The rest of this work is structured as follows. Section 2 reviews the related literature in detail. Section 3 describes the problem setting and formulates a stochastic dynamic programming model. Section 4  summarizes several characteristics of the optimal ordering pattern. We present a heuristic ordering policy and their computation methods in Section 5. A computational study and its results are detailed in Section 6. Finally, Section 7 draws conclusions and outlines future research directions.

\section{Literature review}

An important part of supply chain management is financial flow management\citep{cooper1997supply}. Some works considered the Net Present Value (NPV) of the cash flow in the traditional inventory problems. For example, \cite{helber1998cash} built a NPV model for a multi-item capacity-constrained problem.  \cite{wu2016inventory} made a discounted cash-flow analysis for inventory models with deteriorating items and partial trade credits. Working capital is another aspect that has been investigated by many scholars in recent years. \cite{zeballos2013single} discussed an inventory problem with working capital constraints and short-term debt. \cite{Yuan2018A} took working capital requirements into account in a single-item lot-sizing problem. \cite{luo2019managing} built a working capital maximization model for an inventory problem with trade credit and cash-related costs. There are also some works that addressed cash flow and maximize the long-term survival possibility for start-up firms, see e.g. \cite{archibald2002should}, \cite{tanrisever2012production}, \cite{archibald2015managing}.

Although the aforementioned works have considered cash flow, the available cash in their models does not affect the ordering quantity decisions, hence these works do not consider cash constraints. Among works considering cash constraint, \cite{buzacott2004inventory} noted the importance of jointly considering operational and financial decisions by analyzing a cash-constrained newsvendor model with financing behavior. \cite{dada2008financing} built a Stackelberg model between a bank, manufacturer and cash-constrained retailer. \cite{raghavan2011short} considered a two-level supply chain with a retailer and manufacturer both facing cash constraints. \cite{kouvelis2012financing} gave a detailed discussion of optimal trade credit contracts in a game model. \cite{moussawi2013joint} developed a model for a cash-constrained retailer under delay payments. \cite{Tunca2017Buyer} discussed the role and efficiency of buyer intermediation in supplier financing through a game-theoretical model. \cite{bi2018dynamic} addressed short-term financing in an inventory problem of multi-item setting.

Most of the literature focusing on multi-period stochastic inventory problems with cash constraints consider single-item problems.   \cite{chao2008dynamic} investigated a multi-period single-item self-financing newsvendor problem. They proved the optimal ordering pattern is an approximate base stock policy and presented a simple algorithm to solve the problem for stationary demands.  \cite{gong2014dynamic} extended this by considering short-term financing. \cite{Katehakis2016Cash} analyzed non-stationary demand processes and time-varying interests. \cite{Boulaksil2018} formulated a cash-constrained stochastic inventory model with consumer loans and supplier credits for nanostores.
% and obtained managerial insights by simulating numerical cases.

% They showed that the optimal ordering policy is presented in terms of two thresholds with each cash and inventory state.

Despite considering cash constraints, no fixed ordering cost is embedded in the aforementioned works. The existence of fixed ordering cost makes it difficult to analyze the structure of the optimal ordering policy. \cite{ChenZhang2019} investigated cash flow constraints in a lot-sizing problem. However, their system operates under deterministic demand. Next, we will review some pioneering works dealing with the stochastic inventory problem with fixed ordering costs, which inspires us to develop the ordering policy in this paper.

Without a capacity constraint, \cite{Scarf1960} proved that the optimal ordering policy for the general single-item stochastic lot-sizing problem is $(s, S)$, where $s$ denotes the reorder point and $S$ is the order-up-to level. The optimality is proved through a property called $K$-convexity. Optimal ordering patterns have not been thoroughly characterized for the capacitated stochastic lot-sizing problem. A key finding by \cite{Chen1996X} proved that with stationary demand and fixed capacity, the optimal policy has a $\bf X-Y$ band structure pattern: when the initial inventory level is below $\bf X$, it is optimal to order at full capacity; when the initial inventory level is above $\bf Y$, do not order. \cite{Gallego2000Capacitated} further clarified this structure with CK-convexity and divided the optimal ordering policy into four regions: in two of these regions the optimal policy is completely specified, while it is partially specified in the other two. \cite{shaoxiang2004infinite} extended it for the infinite horizon model. \cite{chao2008optimal} found the optimal policy for a capacitated problem with a special supply chain contract. \cite{ozener2014near} relaxed the capacitated problem with linear holding and penalty cost, Poisson demand, and proposed a heuristic $(s, \Delta)$ policy. \cite{Shi2014Approximation} developed approximation algorithms with worst-case performance guarantee for this problem. There are also some other related works such as \cite{Chan2003A} and \cite{Yang2014Analysis}. While most of the above works considered the capacity constraint as static exogenous,  cash constraint is more complex because actually cash constraint is a function of operational decisions.

%They modeled the problem as a discrete-time Markov chain for stationary demands with an infinite horizon.

For the general $(s, S)$ policy, the computation of the values of $s$, $S$ in each period has been considered to be prohibitively expensive for years. A number of existing works have attempted to solve this problem under stationary demands; see e.g. \cite{federgruen1984efficient}, \cite{zheng1991finding}, \cite{feng2000new}. For non-stationary demands, recent progress has been made by \cite{Bollapragada1999A} and \cite{xiang2018computing}. Nevertheless, those works are for the uncapacitated situation ---- there are no upper bounds on ordering quantity.

% A cash constraint is similar to a general capacity constraint because it also sets an upper bound to the ordering quantity, but also very different in that the cash flow is dynamic and the upper bound in a cash constraint is dynamic and not fixed.

A comparison between some of the aforementioned inventory problems and our research is presented in Table \ref{table:literature}.

\section{Problem setting}
For convenience, the notations adopted in this paper are listed in Table \ref{table:notations}. Other relevant notations will be introduced as needed.

% While there has been sizable research on the characterization of the optimal policy for stochastic lot sizing problems under order quantity capacity constraints, characterizing the structure of the optimal policy for cash constrained stochastic lot sizing problems is an open research direction.

\linespread{1.2}
\makesavenoteenv{tabular}
\makesavenoteenv{table}
\begin{table}[!ht]
\centering
\caption{Comparisons with some other related literature.}
\small
\noindent
\begin{tabular}{lccccc}
\toprule
\multirow{2}*{Authors}  &Demand &\multirow{2}*{Multi items}  & Fixed  & Multi & Capacity \\
& pattern  & & ordering cost & period & constraints\\
\midrule
\cite{helber1998cash}
 &Deterministic &\checkmark &\checkmark &\checkmark
&\\
\specialrule{0em}{2pt}{2pt}

% \cite{wu2016inventory}
%  &\multirow{2}*{ Deterministic} &\multirow{2}*{ \checkmark} & &\multirow{2}*{ \checkmark}
% &\\
% \cite{lusa2012integral}\\
% \specialrule{0em}{4pt}{4pt}

% \cite{grubbstrom2004ordering}
%  &\multirow{2}*{ Deterministic}
%  &
%  &\multirow{2}*{\checkmark} &\multirow{2}*{\checkmark}
%&\\
%\cite{wu2016inventory}\\
\cite{Yuan2018A} &Deterministic & &\checkmark &\checkmark &\\
\specialrule{0em}{2pt}{2pt}

\cite{ChenZhang2019} &Deterministic
&
&\checkmark
&\checkmark
&cash\\
\specialrule{0em}{2pt}{2pt}

\cite{archibald2015managing} &\multirow{2}* {Stochastic} & & &\multirow{2}*{\checkmark} &\\
\cite{luo2019managing}\\
\specialrule{0em}{2pt}{2pt}

\cite{buzacott2004inventory} & \multirow{2}*{Stochastic} & & &&\multirow{2}*{cash}\\
\cite{moussawi2013joint}\\
\specialrule{0em}{2pt}{2pt}

\cite{bi2018dynamic}&Stochastic &\checkmark&&&cash\\
\specialrule{0em}{2pt}{2pt}

\cite{Scarf1960}&Stochastic&&\checkmark\\
\specialrule{0em}{2pt}{2pt}

\cite{Chen1996X} & \multirow{3}*{Stochastic} &  &\multirow{3}*{\checkmark} &&\multirow{3}*{fixed}\\
\cite{Gallego2000Capacitated}\\
\cite{shaoxiang2004infinite}\\
\specialrule{0em}{2pt}{2pt}

% \cite{tanrisever2012production} &
% \multirow{2}*{Stochastic} &  & &\multirow{2}*{\checkmark}&\\
% \cite{archibald2015managing}\\
% \specialrule{0em}{4pt}{4pt}

\cite{chao2008dynamic} &
\multirow{3}*{Stochastic} &  & &\multirow{3}*{\checkmark}&\multirow{3}*{cash}
\\
\cite{gong2014dynamic}&\\
\cite{Katehakis2016Cash}\\
\specialrule{0em}{2pt}{2pt}

\cite{Boulaksil2018}
&Stochastic  &\checkmark
&
&\checkmark
&cash\\
\specialrule{0em}{2pt}{2pt}

This work & Stochastic &  &\checkmark &\checkmark &cash\\
\bottomrule
\end{tabular}\label{table:literature}
\end{table}
\linespread{1.5
}
\normalsize

\subsection{Problem description}
In our problem, a cash-constrained retailer periodically purchases one product from its suppliers and sells the product to customers. Customer demand is non-negative, discrete and independently distributed from period to period. The random demand for each period can be non-stationary and is represented by a variable $\xi_n$. Its probability density function is $\phi_n(\xi)$ and cumulative distribution function is $\Phi_n(\xi)$. Like \cite{shaoxiang2004infinite} and \cite{Chen1996X}, we also assume the maximum demand value is $D_u$. Additionally, we also assume the minimum demand value is $D_b$. Theoretically, $D_b$ can be zero, and $D_u$ can be a very large number. In many real-life cases, extremely large demands are low-probability events and can be neglected. This assumption also guarantees there are finite states in the stochastic dynamic programming of this problem.

\input{parameters}

% have the following two assumptions:
% \begin{enumerate}
%   \item There are finite possible demand values.
%   \item The maximum possible demand quantity is $D_u$, and the minimum possible demand quantity is $0$.
% \end{enumerate}

% %EDITOR: Please ensure that the intended meaning has been maintained in this edit.
% We adopt the first assumption to prove the continuity of the optimality equation in the stochastic dynamic programming, and we adopt the second assumption for the convenience of expressing the $\bf C_1$ and $\bf C_2$ bounds in later sections.

At the beginning of the planning horizon, the retailer has an initial inventory of $x_0$ and an initial cash $R_0$. Since we focus on a retail setting, we assume that customers usually do not wait for back-ordered items and that unmet demand is lost. Moreover, we assume that suppliers require immediate payment and that the order delivery lead time is zero. Excess stock is transferred to the next period as inventory, and the selling back of excess stock is not allowed. At the beginning of any period $n$, the retailer has to decide the order quantity $Q_n$; we let $y_n=x_{n-1}+Q_n$, where $y_n$ denotes the
``order-up-to level.'' The inventory flow function of the retailer can then be expressed as
\begin{equation}
    x_{n}=~\max\{y_{n}-\xi_n,0\}=(y_n-\xi_n)^+\quad\label{eq:IFlow}
\end{equation}

The initial inventory of the retailer is $x_0$, and the initial cash balance of the retailer is $R_{0}$. The selling price of the item is $p$, and the retailer receives payments only when its items are delivered to the customers. A fixed ordering cost $K$ is charged to the retailer when ordering from its suppliers; there is also a variable ordering cost $c$ charged on every ordered unit. We assume that the retailer needs to pay a fixed overhead cost $W$ at the beginning of each period, to operate the business and cover fixed costs such as rents or wages for staffs; while variable inventory holding cost charged on every item unit carried from one period to the next is assumed to be negligible.

% Inventory holding cost is not considered for the retailer. A variable inventory holding cost $h$ is charged on every item unit carried from one period to the next.

In each period $n$, the initial cash is $R_{n-1}$, and the ordering quantity is $y_n-x_{n-1}$. Define $\delta(y_n-x_{n-1})$ as a unit step function to determine whether the retailer makes an order in period $n$: $\delta(y_n-x_{n-1})=1$ when $y_n>x_{n-1}$, $\delta(y_n-x_{n-1})=0$ when $y_n=x_{n-1}$. The total ordering cost in period $n$ is $K\delta(y_n-x_{n-1})+c(y_n-x_{n-1})$, which includes the fixed ordering cost and variable ordering costs. Because of the cash constraint, when the retailer purchases items, its available cash must be greater than the total costs in this period, namely, ordering costs plus overhead cost. This is shown by the constraint below.
\begin{equation}
K\delta(y_n-x_{n-1}) + c(y_n-x_{n-1})+W\leq R_{n-1}\label{con:cashConstraint}
\end{equation}

% At each period, the minimum cash balance for the retailer is $W$, which is the overhead cost for the retailer, such as utility bills or wages for staffs.

The actual sales quantity in period $n$ is $\min\big\{\xi_n,y_n\big\}$, and revenue is $p\min\big\{\xi_n,y_n\big\}$. End-of-period cash $R_n$ for period $n$ is defined as the initial cash balance $R_{n-1}$ plus the period's revenue minus the period's total costs $K \delta(y_n-x_{n-1})+c(y_n-x_{n-1})+W$. Any inventory left at the end of the planning horizon has a salvage value $\gamma$ per unit. Therefore, the full expression of $R_n$ is given by Eq. \eqref{eq:cashFlow}.
\begin{equation}
  R_n=\begin{cases}
  R_{n-1}+  p\min\{\xi_n,y_{n}\}-\left[K \delta(y_n-x_{n-1})+c(y_n-x_{n-1})\right]-W\quad &n\leq N\\
  R_N+\gamma x_N &n=N+1
  \end{cases}\label{eq:cashFlow}
\end{equation}

The main difference between our problem and previous inventory literature about cash constraint in inventory problems, such as \cite{chao2008dynamic} and \cite{gong2014dynamic}, is that we consider a fixed ordering cost $K$ for the retailer. Another difference is that we do not account for the deposit interest of each period's cash balance because the deposit interest from the bank is generally very small compared with the retailer's transaction volume and usually does not exert an effect on its operational decisions.

To avoid triviality, we assume $0\leq \gamma<c<p$. Our aim is to find a replenishment plan that maximizes the expected terminal cash increment, i.e., $E(R_{N+1})- R_0$.

\subsection{Stochastic dynamic programming modeling}
We formulate a stochastic dynamic programming
%EDITOR: Abbreviations and acronyms typically need to be defined only once within the main text. Please consider adhering to this convention.
 model for our problem to analyze its mathematical properties.

\textbf{States.}
The system state at the beginning of period $n$ is represented by initial inventory $x_{n-1}$ and initial cash quantity $R_{n-1}$. Note that since demand is assumed to be discrete and maximum demand value is $D_u$, there are finite states for inventory and cash in our problem.

\textbf{Actions.}
The action in period $n$ is the order-up-to level $y_n$, given initial inventory $x_{n-1}$ and initial cash quantity $R_{n-1}$. From the cash constraint \eqref{con:cashConstraint}, we use a function $B(R)$ to show the upper bound for the ordering quantity.
\begin{align}
  B(R) =\max\left\{0, \frac{R_{n-1}-K-W}{c}\right\}\label{eq:Qbound}
\end{align}

The bounds for $y_n$ are represented by
\begin{equation}
x_{n-1}\leq y_n\leq \overline{y}_n=x_{n-1}+B(R)\label{con:ybound}
\end{equation}

\textbf{State transition function.}
The state transition function for inventory is given by inventory flow equation \eqref{eq:IFlow}, while the state transition function for cash is presented by cash flow equation \eqref{eq:cashFlow}.

\textbf{Immediate profit.}
Let $\Delta R_n$ be the cash increment in period $n$, expressed as
\begin{equation}
\Delta R_n=p\min\big\{\xi_n,y_{n}\big\}-\left[K\delta(y_n-x_{n-1})+c(y_n-x_{n-1})\right]-W \label{eq:deltaR}
\end{equation}

The immediate profit for period $n$ is the expected cash increment $E(\Delta R_n)$ during this period given $R_{n-1}$, $x_{n-1}$ and $y_{n}$. $E(\Delta R_n)$ can be expressed as
\begin{align}
E(\Delta R_n)=\int_{D_b}^{D_{u}}\Delta R_n\phi_n(\xi)d\xi
\end{align}

\textbf{Optimality Equation.}
Define $F_{n}(R_{n-1},x_{n-1})$ as the maximum expected cash increment during periods $n,n+1,\dots,N$ when the initial inventory and cash of period $n$ are $x_{n-1}$ and $R_{n-1}$, respectively. The optimality equation for the problem is expressed as:
\begin{align}
F_{n}(x_{n-1}, R_{n-1})=\max\limits_{x_{n-1}\leq y_{n}\leq \overline{y}_n}\left\{E(\Delta R_n)+ \int_{D_b}^{D_{u}}F_{n+1}((y_{n}-\xi_n)^+, R_{n-1}+\Delta R_n)\phi_n(\xi)d\xi\right\}\quad n=1,2,\dots,N\label{eq:functionalequation}
\end{align}

The boundary equation is:
%EDITOR: Please be consistent with the punctuation of sentences ending in a mathematical equation.
\begin{equation}
F_{N+1}(x_{N}, R_N)=\gamma x_N
\end{equation}

% \begin{lemma}
% $L_n(y)$ is a concave function and reaches its maximum at
% \begin{equation}
%   x_L=\Phi_n^{-1}(\frac{p-c}{p+h})
% \end{equation}. \label{lemma:concave}
% \end{lemma}

% \begin{proof}
% \begin{equation}
%   \frac{d(L_n(y))}{dy}=p-c-(p+h)\int_{0}^y \phi_n(\xi)dt\nonumber
% \end{equation}
%   \begin{equation}
%     \frac{d^2(L_n(y))}{dy^2}=-(p+h) \phi_n(y)<0\nonumber
%   \end{equation}
%   Therefore, $L_n(y)$ is a concave function. We can obtain the maximum point at $x_L$ by setting the first-order derivative to zero.
% \end{proof}

For convenience, we define $L_n(y)$ as follows.
\begin{align}
    L_n(y)=\int_{
    D_b}^{D_{u}}\left[p\min\{\xi_n,y\}-cy\right]\phi_n(\xi)d\xi-W
\end{align}

% in many textbooks (eg. \cite{porteus2002foundations}).

$L_n(y)$ is an obvious concave function. For ease of expression, we suppress $n$, $D_u$, $D_b$ in some expressions and use it only when required. Therefore, $\Delta R$, $L(y)$, $E(\Delta R)$ and $F_n(x, R)$ can be written as:
\begin{align}
  \Delta R = &p\min\{\xi,y\}-K\delta(y-x)-c(y-x)-W\label{eq:deltaR2}\\
  L(y)=&~\int\left[p\min\{\xi,y\}-cy\right]\phi(\xi)d\xi-W\label{eq:Ly2}\\
E(\Delta R)=&\int\left[p\min\{\xi,y\}\right]\phi(\xi)d\xi-K\delta(y-x)-c(y-x)-W=L(y)+cx-K\delta(y-x)\label{eq:EdeltaR2}\\
F_n(x, R)=&
\begin{cases}
\max\limits_{x\leq y\leq x+B(R)}\left\{L(y)+cx-K\delta(y-x)+ \int F_{n+1}((y-\xi)^+, R+\Delta R)\phi(\xi)d\xi\right\} &n\leq N\\
~~\gamma x &n=N+1
\end{cases}
\label{eq:functionalequation2}
\end{align}

%The aim is to maximize the expected cash increment $F_{1}(R_0,x_{0})$ over the planning horizon given initial cash $R_{0}$ and inventory $x_{0}$.

\section{Ordering policy discussion}
In this section, we provide some characteristics about the optimal ordering policy. We first show this by a numerical example, then analyze the optimal ordering policy in the last period, and finally obtain some theoretical expressions with proofs for the general periods.

\subsection{Numerical findings}
Since optimality function $F_n(x, R)$ for the dynamic programming model is two-dimensional, it is difficult to detect its mathematical properties. By fixing $x$ or fixing $R$ and viewing $F_n(x, R)$ as a function of a single decision variable, $F_n(x, R)$ does not show K-convexity as in \cite{Scarf1960} or K-concavity as in \cite{chen2004coordinating}, nor does it show CK-convexity (or CK-concavity) as in \cite{Gallego2000Capacitated} or (C, K)-convexity (or (C, K)-concavity) as in \cite{shaoxiang2004infinite}. This result can be easily confirmed by some numerical examples.

However, we can find some characteristics of the optimal ordering policy. We illustrate these characteristics with the following numerical example: there are 4 periods; demands follow Poisson distributions and the average values are 9, 13, 20, 16; fixed ordering cost $K=20$, unit variable ordering cost $c=1$, overhead costs $W=2$, selling price $p=5$ and unit salvage value $\gamma=0$.

The optimal ordering quantities $Q^\ast(x, R)$ for different initial inventory $x$ and $R$ in the first period are given below:
\vspace{1.5ex}

\input{example1}

As an illustration of how to read the above results, suppose the initial inventory $x$ is 1 and the initial cash $R$ is 35, then, the optimal ordering quantity is 13 units. Several ordering characteristics can be observed directly from this numerical case:
\begin{itemize}
  \item The optimal ordering policy is not the $(s, S)$ type.

  \item When the initial inventory is very large or initial cash is very small, the optimal ordering quantity is always zero. For example, when $x\geq 7$ or $R\leq 27$, $Q^\ast=0$. Furthermore, the ordering threshold of $R$ is different for different initial inventory. When $x=0$, it is optimal to not order when $R\leq 25$; when $x=2$, it is optimal to not order when $R\leq 39$.

%   \item For Case 2, when initial cash is excessive, it is also optimal to not order for some states. This inventory threshold of initial cash $R$ is also related to the initial inventory. For example, when $x=0$, if $R\geq 35$, $Q^\ast=0$; when $x=1$, if $R\geq 33$, $Q^\ast=0$;

  \item When ordering, it is optimal for the retailer to order as close to some specified inventory level as possible. For a fixed initial inventory, there may be several order-up-to levels. We show the different order-up-to levels by drawing a vertical line segment below the ordering quantity values. For example, in this case, the order-up-to levels for $x=0$ are 14, and 24, and the order-up-to levels for $x=2$ are 24.
\end{itemize}

The above case suggests that at least two thresholds exist, $\bf s$ and $\bf C$: if $x\geq \bf s$ or $R\leq \bf C$, do not order; else, order up to some specified level. For ease of analysis, we first investigate the ordering pattern in the last period $N$.

\subsection{Ordering pattern in the last period}
 When $n=N$, from Eq. \eqref{eq:functionalequation2}, its optimality equation is
\begin{align}
  F_N(x, R)=\max\limits_{x\leq y\leq x+B(R)}\left\{L(y)+cx-K\delta(y-x)+\gamma (y-\xi)^+\right\}\label{eq:FNxR}
\end{align}

Recall that $L(y)=p\min\{\xi, y\}-cy-W$. Let
\begin{equation}
    L'(y)=L(y)+\gamma (y-\xi)^+\label{eq:L'y}
\end{equation}

By Eq. \eqref{eq:FNxR} and \eqref{eq:L'y}, the optimal decision for period $N$ is
\begin{equation}
   y^\ast_N(x,R)=\begin{cases}
   y^\ast\quad &\text{if } L'(x)< \max_{x<y\leq B(R)}L'(y)-K\\
  x &\text{if } L'(x)\geq \max_{x<y\leq B(R)}L'(y)-K
\end{cases}\label{eq:optyN}
\end{equation}

Here, we draw the picture of $L'(y)$ of a real numerical example in Figure \ref{fig: PNy}.

\input{drawPNy}

Since $L'(y)$ is also an apparent concave function, it reaches optimal at:
\begin{equation}
  S=\Phi_N^{-1}\left (\frac{p-c}{p-\gamma}\right)\label{eq:SboundN}
\end{equation}

From Eq. \eqref{eq:optyN}, we can obtain a smallest $\bf s$ threshold for period $N$. The $\bf s$ threshold means when initial inventory $x$ is larger than $\bf s$, it is always better to not order. As shown in Figure \ref{fig: PNy}, $L'(s)$ is smaller than $L'(S)$ for a fixed ordering cost $K$.
\begin{equation}
  s=\min\{y\mid L'(y)\geq L'(S)-K\}\label{eq:sboundN}
\end{equation}

When $x<s$, there exists a $\bf C$ threshold. When $R\leq C(x)$, it is optimal for the retailer to not order because the expected profit cannot cover the ordering costs. The maximum $\bf C$ threshold is the cash level for which the retailer's expected profit by ordering equals the fixed ordering cost plus the expected profit of not ordering.
%EDITOR: Please ensure that the intended meaning has been maintained in this edit.
 The $\bf C$ threshold for an initial inventory $x$ can be obtained by the following equation. Recall that $B(R)$ means the maximum ordering quantity under initial cash $R$.
\begin{equation}
  C(x)=\max\{R\mid L'(x)+K\geq L'(x+B(R))\}\label{eq:C1boundN}
\end{equation}

Note that $C(x)$ is state-dependent and related with inventory level $x$. For an initial inventory level $x<s$, its minimum cash requirement for ordering is $C(x)$. Its maximum order-up-to level under initial cash $C(x)$ is $x+B(C(x))$.  The relation of $s$, $S$, $C(x)$ and $L'(y)$ is shown in Figure \ref{fig: PNy}.  Since $L'(y)$ is increasing when $y<S$, it is optimal for the retailer to replenish its inventory level as close to $S$ as possible when $x<s$ and $R>C(x)$.  The optimal ordering policy in period $N$ can easily be expressed by the following equation:
\begin{equation}
  Q^\ast_N=\begin{cases}
  0\quad &\text{otherwise}\\
  \min\{S-x, B(R)\} &x<s, R>C(x)
\end{cases}\label{eq:optOrderN}
\end{equation}

\subsection{Thresholds of \texorpdfstring{$s$}{} and  \texorpdfstring{$C$}{}}
When $n$ is not restricted to the final period $N$, we provide some analytic expressions for the thresholds of $\bf s$ and $\bf C$.

For any period $n$, an $\textbf{s}$ threshold can be obtained by the following equation, where $\Phi_{n\sim N}$ is the cumulative distribution function of total demand from period $n$ to $N$ and $\Phi^-_{n\sim N}$ is its inverse function.
\begin{align}
  \textbf{s}=&\Phi_{n\sim N}^{-1}\left(\frac{p-c}{p-\gamma}\right)\label{eq:Ybound}
\end{align}

\begin{lemma}
  In any period $n$, there exists a threshold $\bf s$ such that if initial inventory $x\geq \textbf{s}$, regardless of the initial cash, there is no need for the retailer to order. The threshold can be defined by Eq. \eqref{eq:Ybound}. \label{lemma:sBound}
\end{lemma}

\begin{proof}
The total demand from period $n$ to period $N$ follows a stochastic distribution.  Let
\begin{equation}
    L'_{n\sim N}(y)=\int[p\min\{\xi, y\}-cy+\gamma (y-\xi)^+]\phi_{n\sim N}(\xi)d\xi-(N-n+1)W \nonumber
\end{equation}

where $\phi_{n\sim N}(\xi)$ is the probability  distribution function of total demand from period $n$ to $N$. Since there is no inventory holding cost in our cost, ordering earlier would not bring extra costs. If the retailer has enough cash, it only needs to order at most once in period $n$ to meet all the later demands, which can save fixed ordering costs and result in a similar newsvendor problem. By the above description and Eq. \eqref{eq:functionalequation2}, we know that
\begin{align}
    F_n(x, R)\leq \max_{y\geq x} L'_{n\sim N}(y)+cx-K\delta(y-x)\\
  \lim_{R\rightarrow \infty}F_n(x, R)=\max_{y\geq x} L'_{n\sim N}(y)+cx-K\delta(y-x)\label{eq:limF(x,R)}
\end{align}

Since $L'_{n\sim N}(y)$ is a concave function and reaches optimal at $\Phi_{n\sim N}^{-1}(\frac{p-c}{p-\gamma})$, by Eq. \eqref{eq:limF(x,R)}, when $x\geq \Phi_{n\sim N}^{-1}(\frac{p-c}{p-\gamma})$, there is no need to order because the available inventory is enough to obtain maximum expected profit. Therefore, Eq. \eqref{eq:Ybound} is a $\bf{s}$ threshold.

\end{proof}

More conservative $\bf s$ thresholds exist. For example, if $x\geq (N-n)D_u$, the initial inventory is sufficient to meet all the demands in later periods; thus, there is no need for the retailer to order. By Lemma \ref{lemma:cashInventoryIncreasing} below, the maximum demand $D_u$ can also be a  $\bf s$ threshold (we omit the proof for brevity). Note that the $\textbf{s}$ threshold defined by Eq. \eqref{eq:Ybound} is also conservative in real numerical tests.

As for the threshold of $\bf C$, we first bring two lemmas about the optimality equation $F_n(x, R)$. Proofs for the two lemmas are in Appendix \ref{app-3}.

% In terms of whether ordering or not in each period, there two situations for cash increment $\Delta R$:
% \begin{align}
%     \Delta RA = &p\min\{\xi,x\}\label{eq:deltaR2-1}\\
%     \Delta RB = &p\min\{\xi,y\}-K-c(y-x)\qquad y>x\label{eq:deltaR2-2}
% \end{align}

% $\Delta RA$ is the cash increment when ordering and $\Delta RB$ is the cash increment when not ordering. Based on the two situations, $GA(y)$ and $GB(y)$ are defined as follows.
% \begin{align}
%   GA(y)=L(y)+ \int_{0}^{D_u}F_{n+1}((y-\xi)^+, R+\Delta RA)\phi(\xi)d\xi\label{eq:GA}\\
%   GB(y)=L(y)+ \int_{0}^{D_u}F_{n+1}((y-\xi)^+, R+\Delta RB)\phi(\xi)d\xi\label{eq:GB}
% \end{align}

% Therefore, $F_n(x, R)$ can be written as follows:
% \begin{align}
%   F_{n}(x, R) = cx+\max\begin{cases}
%   GA(x)\\
%   \max_{x< y\leq x+B(R)}GB(y)-K
% \end{cases}
% \end{align}

% Suppose $\max_{x< y\leq x+B(R)}GB_n(y)=GB_n(y^\ast)$, where $y^\ast\in[x, x+B(R)]$ (This is justified by the continuity of $GB_n(y)$ and $F_n(x,R)$ and we omit the proof here, because it is similar to the proof in \cite{shaoxiang1990}). Then the optimal decision for initial states $x$, $R$ in one period is:
% \begin{equation}
%     y_n^\ast (x, R)=\begin{cases}
%     y^\ast\quad & GB(y^\ast)-GA(x)>K\\
%     x\quad & GB(y^\ast)-GA(x)\leq K
%     \end{cases}\label{eq:optDecision}
% \end{equation}

% Eq. \eqref{eq:optDecision} means that if the retailer makes order in one period, its expected profit surplus compared with the situation of not ordering must larger than the fixed ordering cost $K$.

\begin{lemma}
  For any period $n$ and fixed $x$, $F_n(x, R)$ is non-decreasing in $R$.\label{lemma:non-decreasing}
\end{lemma}

Note that when $R$ is sufficiently large, there is no ordering quantity constraint for the retailer, and the problem is the same as the traditional problem in \cite{Scarf1960}. In this situation, for fixed $x$, $F_n(x, R)$ is constant.

\begin{lemma}
  For any period $n$, $\forall b>0$,  $F_n(x,~R+K+cb)+K+cb\geq F_n(x+b, ~R)$. \label{lemma:cashInventoryIncreasing}
\end{lemma}

Lemma \ref{lemma:cashInventoryIncreasing} implies that the expected cash increment for state $(x, R+K+cb)$ is larger than $(x+b, R)$. Thus, for two initial states ($x_1$, $R_1$) and  ($x_2$, $R_2$), if one state has less initial inventory $(x_1<x_2)$ but more initial cash, and the extra cash is sufficient to replenish the initial inventory to equal that of the other state ($R_1- R_2\geq K+c(x_2-x_1)$), then this state will result in a larger expected cash increment at the end of the planning horizon.

When initial inventory $x$ is below the $\textbf{s}$ threshold, one may order or not depending on available cash. If the initial cash $R$ is small, the retailer cannot order up to its desired inventory level to make revenue cover the ordering costs. This situation leads to the existence of a $\bf C$ threshold. When initial cash $R\leq \bf C$, it is better for the retailer to not order. We develop two forms of threshold $\bf C$.
\begin{align}
\textbf{C}=& K\label{eq:Cbound1}\\
  \textbf{C}=&\max\{R\mid x+B(R)\leq D_b, pB(R)\leq K+cB(R)\}\label{eq:Cbound}
\end{align}

\begin{lemma}
  For any period $n$, given an initial inventory $x$, if $x$ is not a $\bf{s}$ threshold, there exists a $\bf{C}$ threshold. If initial cash $R\leq \bf{C}$, it is better for the retailer to not order. The expression of Eq. \eqref{eq:Cbound1} or the solution of \eqref{eq:Cbound} can be a $\bf{C}$ threshold.\label{lemma: Cbound}
\end{lemma}

\begin{proof}

Eq. \eqref{eq:Cbound1} is an apparent $\bf C$ threshold. If initial cash is less than $K$, it is always better for the retailer to not order because the retailer does not even have sufficient cash to cover the fixed ordering cost. For the initial cash $R$ defined by Eq.  \eqref{eq:Cbound}, because $x+B(R)\leq D_b$, it is better to order at full capacity since this ordering will result in the greatest cash increment.  Thus, there are two scenarios:
\begin{enumerate}[topsep=2pt, itemsep=2pt,partopsep=2pt]
  \item[(a).] do not order.
  \item[(b).] order $B(R)$ units.
\end{enumerate}

The rest of the proof is similar to that of Lemma \ref{lemma:cashInventoryIncreasing}. Since $x+B(R)\leq D_b$, $x\leq x+B(R)\leq \xi$. The cash increments for the two scenarios are as follows.
\begin{align}
  \Delta R^a=&~p\min\{x,\xi\}=px\nonumber\\
  \Delta R^b=&~p\min\{x+B(R),\xi\}-K-cB(R)=p(x+B(R))-K-cB(R)\nonumber
\end{align}

From Eq. \eqref{eq:Cbound}, $\Delta R^a\geq \Delta R^b$ because $pB(R)\leq K+cB(R)$. Define $f_n(x)$ and $f_n(x+B(R))$ as the expected cash increments of the two scenarios.
\begin{align}
f_n(x)=&~E(\Delta R^a)+\int F_{n+1}(0, R+\Delta R^a)\phi(\xi)d\xi\nonumber\\
f_n(x+B(R))=&~E(\Delta R^b)+\int F_{n+1}(0, R+\Delta R^b)\phi(\xi)d\xi\nonumber
\end{align}

Since $\Delta R^a\geq \Delta R^b$, $E(\Delta R^a)\geq E(\Delta R^b)$ and $R+\Delta R^a\geq R+\Delta R^b$. By lemma \ref{lemma:non-decreasing}, $F_{n+1}(0, R+R^a)\geq F_{n+1}(0, R+R^b)$.

Therefore, $f_n(x)\geq f_n(x+B(R))$, and it is better for the retailer to not order.

\end{proof}

Note that the \textbf{$C$} thresholds defined by Eq. \eqref{eq:Cbound} or Eq. \eqref{eq:Cbound1} are conservative. Moreover, Eq. \eqref{eq:Cbound} shows the \textbf{$C$} threshold may be state-dependent and related to initial inventory $x$ for each period.

\section{A heuristic ordering policy}

With the existence of the $\bf s$ threshold and  $\bf C$ threshold, the optimal ordering pattern is still very complex, except in the last period. Based on the above analysis and numerical tests, we propose a heuristic ordering policy.
\vskip 3pt

\noindent\textbf{$\bf (s, C(x), S)$ policy}:
\begin{equation}
  Q^\ast=\begin{cases}
  \min\{B(R), S-x\}\quad & \text{if~} x<s~\text{and}~ R>C(x)\\
  0\quad &\text{otherwise}
\end{cases}
\end{equation}

Note that the values of $s$ and $S$ can vary by period, and the values of $C$ are state-dependent and related to initial inventory $x$. $(s, C(x), S)$ policy indicates that if the initial inventory is less than $s$ and initial cash is above $C(x)$, the inventory level should be replenished to as close to $S$ as possible; otherwise, do not order.

% Since the $\bf C2$ bound does not exist in some demand patterns, we also propose a heuristic $(s, C1, S)$ ordering policy.
% \vskip 3pt

% \noindent\textbf{$\bf (s, C1, S)$ policy}:
% \begin{equation}
%   Q^\ast=\begin{cases}
%   \min\{B(R), S-x\}\quad & \text{if~} x<s~\text{and}~ R>C1(x)\\
%   0\quad &\text{otherwise}
% \end{cases}
% \end{equation}

% For policy $(s, C1, S)$, there is only one inventory-related $C1$ for the initial cash threshold.

% To develop an easy-to-implement policy, we also provide $(s, \overline{C}, S)$ policy, \textcolor{red}{where $\overline{C}$ is not state-dependent, not related to initial inventory $x$ and there is only one value for $\overline{C}$ in each period}.
% \vskip 3pt

% \noindent\textbf{$\bf (s, \overline{C}, S)$ policy}:
% \begin{equation}
%   Q^\ast=\begin{cases}
%   \min\{B(R), S-x\}\quad & \text{if~} x<s~\text{and}~ R>\overline{C}\\
%   0\quad &\text{otherwise}
% \end{cases}
% \end{equation}

% In this policy, there is only one pair of $s$ and $S$ for a period. However, in real numerical tests, there may be several levels of $s$ and $S$ within a period. This is one of the reasons that this policy is heuristic. *** LET'S LEAVE THIS COMMENT OUT. ***

%For $(s, \tilde{C}, S)$ policy, $\tilde{C}$ is not related with initial inventory $x$, and there is only one value for $\tilde{C}$ in each period.

\subsection{Computation for \texorpdfstring{$s$}{}, \texorpdfstring{$C$}{} and \texorpdfstring{$S$}{}}
To compute values of $s$, $C$ and $S$
 %EDITOR: Please ensure that the intended meaning has been maintained in this edit.
 we provide two methods here: the first involves searching the results of stochastic dynamic programming (SDP) with some  heuristic techniques; the second is to approximate the problem with a mixed-integer model and solve the problem with some heuristic equations.

\subsubsection{Obtaining \texorpdfstring{$s$}{}, \texorpdfstring{$C(x)$}{}, and \texorpdfstring{$S$}{}  via SDP}

Let $Q^\ast_n (x, R)$ represent the optimal ordering quantity for initial inventory $x$ and initial cash $R$ at period $n$ in the SDP results. The default values for $s$ and $S$ are zero, the default values for $C$ are fixed ordering cost $K$.

For period $n=1$, if $Q_1^\ast(x, R)>0$, $s=x+1$, $S=x+Q^\ast(x, R)$, $C(x)=0$, if $Q_1^\ast(x, R)=0$, $s=x$, $S$, $C(x)$ can be fixed arbitrarily. When $n=N$, the values of $s$, $S$ and $C(x)$ can be obtained by Eq. \eqref{eq:sboundN}, \eqref{eq:SboundN}, \eqref{eq:C1boundN}, respectively.

When $1<n<N$, the method for obtaining $s$, $S$, $C$ is given below.

\textbf{Find $s$:} $s$ is selected as the minimum initial inventory level that results in not ordering for all cash states in the SDP results.

\textbf{Find $S$:} in a given period, when $x<s$, there may be more than one order-up-to level, as shown by our previous numerical case. We adopt a heuristic step by selecting the most frequent order-up-to level as $S$.

\textbf{ Find $C(x)$:} for initial inventory $x$ in a given period and $x<s$, $C(x)$ is the maximum value of initial cash that results in not ordering in the SDP results.

% \textbf{Compute $\overline{C}$:} in a given period, $\overline{C}$ is computed as the average value of initial cash that results in not ordering.

The pseudo-code of this method is given in Algorithm \ref{algorithm1} of Appendix \ref{app-1}. Since this method is based on the results of SDP, it enumerates all the states in each period. This is very time-consuming when the planning horizon is large. For example, in our numerical tests, if the length of the planning horizon is 15, it may take more than one hour to compute one case. So this method is only suitable for the cases of small planning horizons.

\subsubsection{Obtaining \texorpdfstring{$s$}{}, \texorpdfstring{$C(x)$}{}, and \texorpdfstring{$S$}{} values by mixed-integer programming (MIP) approximation}

 The method here adopts a similar idea of ``static-dynamic uncertainty" strategy in \cite{bookbinder1988strategies}.  Replenishment periods are first computed at the beginning of the planning horizon by a MIP model; then we decide the values of $s$, $S$ and $C(x)$  by approximating a newsvendor problem in each replenishment cycle.

The expected demand, end-of-period cash, and end-of-period inventory for period $n$ are represented by $\tilde{d}_n$, $\tilde{R}_n$ and $\tilde{x}_n$, respectively. The mixed-integer programming (MIP) model for the cash-constrained problem is as follows. The major role of the MIP model is to obtain a heuristic ordering plan by determining the values of binary variable $z_n$ that represents whether to order or not in period $n$.
\begin{alignat}{2}
  &\max\qquad &&\tilde{R}_N+\gamma \tilde{x_N}-R_0\label{eq:objectMIP}\\
  &\text{s.t.}&&n=1,2,\dots,N\nonumber\\
  & && \tilde{R}_{n-1}\geq Kz_n+c(S_n-\tilde{x}_{n-1})+W\label{con:cashConMIP}\\
  & &&\tilde{R}_n=\tilde{R}_{n-1}+p(S_n-\tilde{x}_n)-Kz_n-c(S_n-\tilde{x}_{n-1})-W\label{eq:cashFlowMIP}\\
  & && \tilde{x}_{n-1}\leq S_n\label{con:IOrderUpTo}\\
  & && S_n-\tilde{d}_n=  \tilde{x}_n\label{con:IDOrderUpTo}\\
  & && \tilde{x}_0=x_0,  \tilde{R}_0=R_0\label{eq:iniIR}\\
  & &&\tilde{x}_n\geq 0,z_n\in\{0,1\},S_n\geq 0\label{con:xywMIP}
\end{alignat}

In the MIP model, decision variable $S_n$ is the order-up-to level in period $n$. Objective function Eq. \eqref{eq:objectMIP} is the expected cash increment in the planning horizon. The cash constraint is represented by \eqref{con:cashConMIP}, and expected cash flow is given by \eqref{eq:cashFlowMIP}. Constraint \eqref{con:IOrderUpTo} indicates that the order-up-to level $S_n$ should not be less than the period's expected initial inventory. Constraint \eqref{con:IDOrderUpTo} represents the relationship of $S_n$ with $\tilde{x}_n$ and $\tilde{d}_n$. Constraint \eqref{eq:iniIR} is the value of initial inventory and initial cash. Constraint \eqref{con:xywMIP} ensures the nonnegativity and/or the binary nature of the decision variables.

For the last period, $S$, $s$ and $C$ are obtained by equations \eqref{eq:SboundN}, \eqref{eq:sboundN} and \eqref{eq:C1boundN}. For $n<N$, we compute them  in the following.

\begin{enumerate}[leftmargin = 10 pt]
  \item Compute $S$.\\
  Ordering cycles can be obtained from the results of $z_n$ in the MIP model. For example, if $z=\{1, 0, 0, 1, 0\}$, there are two ordering cycles in this ordering plan: period 1$\sim$3 and period 4$\sim$5. This can by shown by Figure \ref{fig:orderCycle}.
  \begin{figure}[ht]
  \centering
  \includegraphics[scale=1.0]{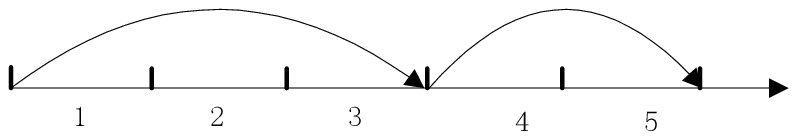}
  \caption{A sketch of two consecutive ordering cycles}\label{fig:orderCycle}
  \end{figure}

  Because the demands in each period are independent, the demands in an ordering cycle follow a joint distribution. Assume that the beginning period for an ordering cycle is $n_1$ and that the end period for an ordering cycle is $n_2$ ($n_1<n_2\leq N$). Let $\phi_{n_1\sim n_2}$ represent the joint probability density function from period $n_1$ to period $n_2$ and $\Phi_{n_1\sim n_2}$ represent the joint cumulative distribution function. Then, in an ordering cycle, the expected cash increment without fixed ordering cost and overhead cost from period $i$ ($n_1\leq i\leq n_2$) to $n_2$ is similar to a newsvendor model, that is:
\begin{equation}
    L_{i\sim n_2}(y)=
    \begin{cases}
    \int\left[p\min\{\xi,y\}-cy\right]\phi_{i\sim n_2}(\xi)d\xi\quad & n_2<N\\
    \int\left[p\min\{\xi,y\}-cy+\gamma(y-\xi)^+\right]\phi_{i\sim n_2}(\xi)d\xi& n_2=N
    \end{cases}\label{eq:Ly2-2}
\end{equation}

Since $L_{i\sim n_2}(y)$ is concave, it reaches its maximum at $y^\ast_{i\sim n_2}=\Phi_{i\sim n_2}^{-1}(\frac{p-c}{p})$ ($n_2<N$) or $y^\ast_{i\sim n_2}=\Phi_{i\sim n_2}^{-1}(\frac{p-c}{p-\gamma})$ ($n_2=N$). $S_i$ ($n_1\leq i\leq n_2$) can be computed by the following heuristic equation:
\begin{equation}
  S_i=\min\{\tilde{x}_{i-1}+B(\tilde{R}_{i-1}),y_{i\sim n_2}^{\ast}\} \label{eq:Si}
\end{equation}

In Eq. \eqref{eq:Si}, $\tilde{x}_{i-1}+B(\tilde{R}_{i-1})$ is the approximate maximum order-up-to level for period $i$. The relation of $S_i$, $y^\ast$ and $L_{i\sim n_2}$ for a numerical case is shown in Figure \ref{fig:Lny}, where the red line represents $L_{i\sim n_2}(y)$ and $S_i=\tilde{x}_{i-1}+B(\tilde{R}_{i-1})$ in this case.
\input{drawLny}

\item Compute $s$.\\
  In an ordering cycle $n_1\sim n_2$, we use a heuristic step to compute $s_i$ ($n_1\leq i\leq n_2$) by the following equation.
  \begin{equation}
    s_i=\min\{y\mid L_{i\sim n_2}(y)\geq L_{i\sim n_2}(S_i)-K, 0\leq y\leq S_i\} \label{eq:si}
  \end{equation}

The idea behind Eq. \eqref{eq:si} is that $s_i$ is the smallest $y$ such that the gap between $L_{i\sim n_2}(y)$ and $L_{i\sim n_2}(S_i)$ is larger than the fixed ordering cost $K$. It is also shown in Figure \ref{fig:Lny}.

  \item Compute $C(x)$.\\
When initial inventory $x<s_i$, for period $i$ ($n_1\leq i\leq n_2$), $C_i(x)$ is computed by the following heuristic equation.
\begin{equation}
  C_i(x)=\begin{cases}
  \max\{R\mid L_i(x)\leq L_i(x+B(R))-K , 0\leq x\leq S\}\quad &L_i(S)\geq K\\
  M &L_i(S)< K
\end{cases}\label{eq:MIPC}
\end{equation}

Note that $C_i(x)$ is related to initial inventory $x$. In Eq. \eqref{eq:MIPC}, $L_i()$ is the one-period expected cash increment without overhead cost for period $i$, and $S$ is its maximum point, which is $\Phi_i^-(\frac{p-c}{p})$. The relationship of $\tilde{C}_i(x)$ and $x$ is shown in Figure \ref{fig:Liy}. Eq. \eqref{eq:MIPC} suggests that if the one-period maximum cash increment is less than the fixed ordering cost, the retailer should not order and $C_i(x)$ is a large number $M$; otherwise, $C_i(x)$ is the maximum cash such that the one-period expected cash increment minus the ordering costs of ordering full capacity is larger than the cost of not ordering.

\input{drawLiy}

\end{enumerate}

The pseudo-code for this method is shown in Algorithm \ref{algorithm2} of Appendix \ref{app-1}. Comparing with obtaining values of $s$, $C(x)$ and $S$ by SDP, it can solve the problem very fast. For numerical cases with 20 periods, the MIP method can get values of $s$, $C(x)$ and $S$ in seconds during our tests.

\subsection{An illustrative example}
To further illustrate the policy and our computational methods, we present a numerical example: there are 4 periods, and the demands in each period follow a Poisson distribution. The mean demand in each period is [20, 7, 2, 14]. Other relevant parameter values are fixed ordering cost $K=24$, unit variable ordering cost $c=1$, overhead cost $W=2$, selling price $p=4$, unit salvage value $\gamma = 0$, initial inventory $x=0$, and initial cash $R=33$. The results of different methods are given in Table \ref{table:heurisicExample}.

\linespread{1}
\begin{table}[!ht]
  \centering
  \caption{An illustrative example of computing $s$, $C(x)$, and $S$}\label{table:heurisicExample}

  \begin{tabular}{lccccc}
    \toprule
    & $n=1$ &$n=2$ &$n=3$ &$n=4$ &Optimality gap\\
    \midrule
    \multirow{9}*{SDP method} &$s=1$ &$s=1$ &$s=2$ &$s=7$ &\multirow{9}*{0.00\%} \\
    \specialrule{0em}{1pt}{1pt}
    &\multirow{7}*{$C(x=0)=0$} &\multirow{7}*{$C(x=0)=34$}  &$C(x=0) =32$ &$C(x=0)=26$ \\
    & &  &$C(x=1) =39$ &$C(x=1)=30$\\
    & & & &$C(x=2)=30$\\
    & & & &$C(x=3)=30$\\
    & & & &$C(x=4)=30$\\
    & & & &$C(x=5)=30$\\
     & & & &$C(x=6)=31$\\
    &$S=7$       &$S=9$ &$S=19$ &$S=17$\\

%     \specialrule{0em}{6pt}{6pt}
%     \multirow{6}*{$(s, \overline{C}, S)$ by SDP} &$s=1$ &$s=1$ &$s=2$ &$s=7$ &\multirow{6}*{0.00\%}\\
%     &\multirow{4}*{$\overline{C}=0$} &\multirow{4}*{$\overline{C}=21$}  &\multirow{4}*{$\overline{C} =17.6$} &\multirow{4}*{$\overline{C}=30$}\\
%     & & & &\\
%     & & & &\\
%     & & & &\\
%     \specialrule{0em}{1pt}{1pt}
%   &$S=7$       &$S=9$ &$S=19$ &$S=17$\\

    \specialrule{0em}{6pt}{6pt}
    \multirow{9}*{MIP method} &$s=3$ &$s=2$ &$s=0$ &$s=7$ &\multirow{9}*{3.49\%}\\
    \specialrule{0em}{1pt}{1pt}
    & &\multirow{7}*{$C(x=0)=32$}  &\multirow{7}*{$C =0$} &$C(x=0)=30$\\

    & & & &$C(x=1)=30$\\
    &$C(x=0)=26$& & &$C(x=2)=30$\\
    &$C(x=0)=30$ & & &$C(x=3)=30$\\
    &$C(x=0)=30$&&& $C(x=4)=30$\\
     &&&& $C(x=5)=30$\\
      &&&& $C(x=6)=31$\\
    \specialrule{0em}{1pt}{1pt}
    &$S=9$       &$S=11$ &$S=3$ &$S=17$\\
    \bottomrule
  \end{tabular}
\end{table}
\linespread{1.5}

As Table \ref{table:heurisicExample} shows, the SDP method is near-optimal, while the MIP method has an optimality gap 3.49\%.

% The values of $s$ and $S$ are the same for the first two policies, while for the MIP method, $s$ and $S$ may be different.

%In the first and last period, the values of $s$, $C1$, $C2$ and $S$ are directly computed. In period 3, since $s=0$, it is always optimal not to order at this period, and values of $C1$ are default fixed ordering cost 20, values of $C2$ are default $M$. In period 2, since $s=2$, for policy $(s, C1, C2, S)$, $C1$ and $C2$ have different values for $x=0$ and $x=1$; for policy $(s, C1, S)$, $C1$ have different values for $x=0$ and $x=1$.

% &$C1(x=0)=26$\\
% &&&&$C1(x=1)=26$\\
% &&&&$C1(x=2)=26$\\

\section{Computational study}
In this section, we present an extensive computational study on large test beds to investigate the effectiveness of the ordering policy and computational methods we proposed.

\subsection{Test bed}

The test beds are adopted from \cite{rossi2015piecewise} with some modifications. There are 10 demand patterns for numerical analysis: 2 life cycle patterns (LCY1 and LCY2), 2 sinusoidal patterns (SIN1 and SIN2), 1 stationary pattern (STA), 1 random pattern (RAND), and 4 empirical patterns (EMP1, EMP2, EMP3, EMP4). To ensure that SDP can solve these instances in reasonable time, we rescale the original demand values, select 10 successive periods for testing, and make cash and inventory be integers. The expected demands for different patterns are given in Table \ref{table:demand datas} and Figure \ref{fig:expecteddemand} in Appendix \ref{app-2}.

Fixed ordering cost $K$ takes values in (10, 15, 20);  unit variable ordering cost $c$ is 1, overhead cost $W=2$ and unit salvage value $\gamma=0$; price $p$ takes values in (5, 6, 7), which results in product margin taking values in (4, 5, 6); and initial cash capacity takes values in (6, 8, 10), which means initial cash $B_0$ is sufficient to order 6, 8 or 10 items. For example, if initial cash capacity is 6, initial cash $B_0=K+c*6$. The initial cash capacity settings guarantee that the retailer does not have enough cash in the first period for all the cases, which makes these cases different from uncapacitated situations. Demands follow a Poisson distribution and are independent in each period. Therefore, there are 270 numerical cases in total for each ordering policy.

The computational studies are coded in Java and run on a desktop computer with an Intel (R) Core (TM) i5-7500 CPU, at 3.40 GHz, 8 GB of RAM, and 64-bit Windows 7 operating system.

 % Our code for the problem is available on Github\footnote{https://github.com/RobinChen121/Stochastic-Lot-Sizing}.

\subsection{Results analysis}
We compare this heuristic policy with the optimal results directly obtained by stochastic dynamic programming (SDP). After the values of $s$, $C(x)$ and $S$ of the different methods are obtained, we simulate them in 100,000 stochastic demand samples to obtain the expected final values and compare the gaps with respect to SDP. The average optimality gap (AVG) and maximum optimal gap (MAX) are adopted as performance metrics. Comprehensive comparison results are shown in Table \ref{table:NumericalResults}.

\input{computeResults}

As shown in the table, policy $(s, C(x), S)$ performs very well in terms of both the maximum gap (0.65\%) and the average gap (0.01\%) of the SDP method. The MIP method has a maximum gap of approximately 17.66\% and an average gap of approximately 2.05\%. Among all 270 numerical cases in the test bed, there are no instances for the SDP method that have an optimality gap greater than 1\%. For the MIP method, 182 numerical cases have a gap exceeding 1\%.

% However, in general, the average gaps for the three policies are all less than 1\%. There exists some special numerical case has a maximum gap of 30.87\% for $(\tilde{s}, \tilde{C}, \tilde{S})$ in the MIP model, but the average gap is only 1.16\%.

%We list the 0.99 confidence intervals for the general average optimality gaps of different methods here: $\pm 0.01\%$, $\pm 0.14\%$, $\pm 0.01\% $, $\pm 0.02\% $.

When fixed ordering cost increases, the numerical tests show that the MIP method performs worse with obvious increasing optimality gaps (average gap changes from 1.35\% to 2.87\%, maximum gap changes from 4.85\% to 17.66\%) and there are also slightly increasing optimality gaps for the SDP method (average gap changes from 0.00\% to 0.03\%). However, with an increasing margin of the product or cash capacity, both methods perform better with lower optimality gaps.

The results also suggest that demand patterns may have a large effect on the performance of the two methods. For example, the maximum gaps and average gaps all grow large in demand pattern EMP1, EMP2, EMP3, EMP4 for policy the MIP method. But for the SDP method, the demand patterns do not have a clear influence to its average gaps. Moreover, for the stationary demand pattern (STA) which has the least demand fluctuation, both methods show good performance with low maximum gap and average gap.

%for all four methods show a random demand pattern (RAND), in which there are some lower expected demands and some much higher expected demands during the planning horizon.
%EDITOR: Please ensure that the intended meaning has been maintained in this edit.

In terms of computational efficiency, the average running time for the SDP method exceeds 300 seconds. However, the MIP approximation method is very fast, with an average running time of less than 1 second. Adding that its average gap is very low (2.05\%), the MIP method performs very well.

\section{Conclusions}
Cash flow management is a key concern for many small businesses. In this paper, we consider a cash-constrained retailer with fixed ordering costs to maximize its final expected cash increment. We find that the optimal ordering policy for this problem is complex but has a $\bf s-C$ structure: when initial inventory is over the $\bf s$ threshold, do not order; when initial inventory is lower than the $\bf s$ threshold, but initial cash is less than the $\bf C$ threshold, do not order either. We propose a heuristic ordering policy for this problem: $(s, C(x), S)$. Computational experiments demonstrate that $(s, C(x), S)$ performs well featuring a very small optimality gap. A heuristic method that approximates this problem as a mixed-integer linear model and newsvendor problem is also developed to compute the values of $ s$, $C(x)$ and $S$. Tests show that this method can solve the problem rapidly with narrow optimality gaps. Future research could extend this work in several directions: the first extension is to consider lead time (time lag) in the problem; the second possible extension is to investigate the cash flow constraints in the multi-item stochastic inventory problem.

\section*{Acknowledgments}
The research presented in this paper is supported by the Fundamental Research Funds for the Central Universities of China under Project SWU1909738, National Natural Science Foundation of China under projects
 71571006.

\bibliography{liter}
\bibliographystyle{plainnat} % plain not fits for author year citation

\clearpage
\begin{appendices}
\renewcommand\sectionname{Appendix}
\section{Proofs for Lemma \ref{lemma:non-decreasing} and Lemma 3}\label{app-3}

\subsection{Proof for Lemma \ref{lemma:non-decreasing}}
\begin{proof}

  When $n=N+1$,  $F_{N+1}(x, R) = \gamma x$. It is intuitively clear that $F_{N+1}(x, R)$ is non-decreasing in $R$.

  Now, suppose this property holds for period $n+1$; we prove the results for period $n$.

  When $R$ is increased to $R'$, decision variable $y$ has a larger domain. Assume the optimal $y$ for initial cash $R$ is $y^\ast$. For initial cash $R'$, let its decision $y^{'\ast}=y^\ast$. Thus, $(y^{'\ast}-\xi)^+=(y^\ast-\xi)^+$. From Eq. \eqref{eq:deltaR2} and \eqref{eq:EdeltaR2}, $\Delta R'=\Delta R$, $E(\Delta R)=E(\Delta R')$, $L(y^\ast)+cx-K\delta(y^\ast-x)=L(y^{\ast '})+cy^{\ast '}-K\delta(y^{\ast '}-x)$. In Eq. \eqref{eq:functionalequation2}, because $F_{n+1}((y^{\ast'}-\xi)^+, R'+\Delta R)$ is not less than $F_{n+1}((y^\ast-\xi)^+, R+\Delta R)$ by supposition, $F_n(x, R')$ is not less than $F_n(x, R)$.

Lemma \ref{lemma:non-decreasing} is thus proved by induction.

\end{proof}

\subsection{Proof for Lemma \ref{lemma:non-decreasing}}

\begin{proof}

  For two scenarios of states $(x, R+K+cb)$ and $(x+b, R)$, let $\Delta R^a$ and $\Delta R^b$ represent the one-period cash increment of the two scenarios, respectively.  From Eq. \eqref{eq:functionalequation2},

  Scenario (a):
  \begin{align}
    F_n(x,~R+K+cb)=\max\limits_{x\leq y\leq x+B(R+K+cb)}\left\{E(\Delta R^a)+ \int_{0}^{D_u}F_{n+1}((y-\xi)^+, R+K+cb+\Delta R^a)\phi(\xi)d\xi\right\}\label{eq:Sca}
  \end{align}

Scenario (b):
\begin{align}
  F_n(x+b,~R)=\max\limits_{x+b\leq y\leq x+b+B(R)}\left\{E(\Delta R^b)+ \int_{0}^{D_u}F_{n+1}((y-\xi)^+, R+\Delta R^b)\phi(\xi)d\xi\right\}\label{eq:Scb}
\end{align}

By Eq. \eqref{eq:Qbound}, $B(R+K+cb)=\max\{0, R/c+b\}=b+B(R)$. The domain for Scenario (a) is $x\leq y\leq x+b+B(R)$, and the domain for Scenario (b) is $x+b\leq y\leq x+b+B(R)$. The domain for Scenario (a) is larger than that of Scenario (b). For any decision $y^b$ in Scenario (b), let $y^a=y^b$. Then,
\begin{align}
  \Delta R^a=&~p\min\{\xi,y^a\}-K\delta(y^a-x)-c(y^a-x)-W\nonumber\\
  \Delta R^b=&~p\min\{\xi,y^b\}-K\delta(y^b-x-b)-c(y^b-x-b)-W\leq\Delta R^a+K+cb\nonumber
\end{align}

  Therefore, $E(\Delta R^a)+K+cb \geq E(\Delta R^b)$. The end-of-period cash of period $n+1$ for the two scenarios is:
\begin{align}
  R_{n+1}^a=&~R+K+cb+\Delta R^a\nonumber\\
  R_{n+1}^b=&~R+\Delta R^b\leq R_{n+1}^a\nonumber
\end{align}

Since $y^a=y^b$, $(y^a-\xi)^+=(y^b-\xi)^+$. By Lemma \ref{lemma:non-decreasing}, $F_{n+1}((y^a-\xi)^+, R_{n+1}^a)\geq F_{n+1}((y^b-\xi)^+, R_{n+1}^b)$. Additionally, because $E(\Delta R^a)+K+cb \geq E(\Delta R^b)$, by Eq. \eqref{eq:Sca} and Eq. \eqref{eq:Scb}, we can conclude that $F_n(x,~R+K+cb)+K+cb\geq F_n(x+b, ~R)$.

\end{proof}

\section{Algorithms to compute \texorpdfstring{$s$}{}, \texorpdfstring{$C(x)$}{} and \texorpdfstring{$S$}{}}\label{app-1}

\begin{algorithm}
\caption{Obtaining $s, C, S$ via SDP}\label{algorithm1}
\KwData{Optimal ordering quantity $Q_n^\ast(x, R)$ for each state $x$, $R$ in each period $n$ from SDP, total states for period $n$ is represented by set $\Omega_n$ }
\KwResult{Values of $s, C(x), S$ in each period }
Initialize:~~$s\leftarrow 0$,~~ $S\leftarrow 0$, ~~$C(x)\leftarrow K$.\;

\For {$n=1$ \emph{\KwTo} $N$}{
\uIf{$n=1$}{
\eIf{$Q^\ast_1(x,R)>0$}{$s\leftarrow x+1$,~~ $S\leftarrow x+Q_1^\ast(x, R)$\;
  $C(x)\leftarrow0$\;}
  {$s\leftarrow x$\;
  $S, ~~C(x)$ can be fixed arbitrarily\;}
}
\uElseIf{$n=N$}{
$s$ is computed by Eq. \eqref{eq:sboundN}, ~~$S$ is computed by Eq. \eqref{eq:SboundN}, ~~$C(x)$ is computed by Eq. \eqref{eq:C1boundN}\;
}
 \Else{
 $s\leftarrow\min\{x\mid Q_n^\ast(x, R)=0, ~~ (x,R)\in\Omega_n\}$\;
 $S\leftarrow$ most frequent order-up-to level in this period\;
 $C(x)\leftarrow \max\{R\mid Q^\ast_n(x, R)=0, x<s\}$\;
 }
}
\end{algorithm}

\begin{algorithm}[H]
\caption{Obtaining $s, C, S$ by MIP approximation}\label{algorithm2}
\KwData{parameter values of the problem}
\KwResult{Values of $s$, $C(x)$, $S$}
Initialize:~~$s\leftarrow 0$,~~ $S\leftarrow 0$, ~~$C(x)\leftarrow K$\;
Compute the MIP model and get a heuristic ordering plan represented by values of $z_n$\;

\ForEach{ordering cycle in the ordering plan}{
\ForEach{period in the ordering cycle}{
compute $S$ by Eq. \eqref{eq:si}\;
compute $s$ by Eq. \eqref{eq:Si}\;
compute $C(x)$ by Eq. \eqref{eq:MIPC}\;
}
}
\end{algorithm}

\clearpage
\section{Demand patterns in the test bed}\label{app-2}

\input{demandValues}

\begin{figure}[!ht]
\centering
\includegraphics[scale=0.45]{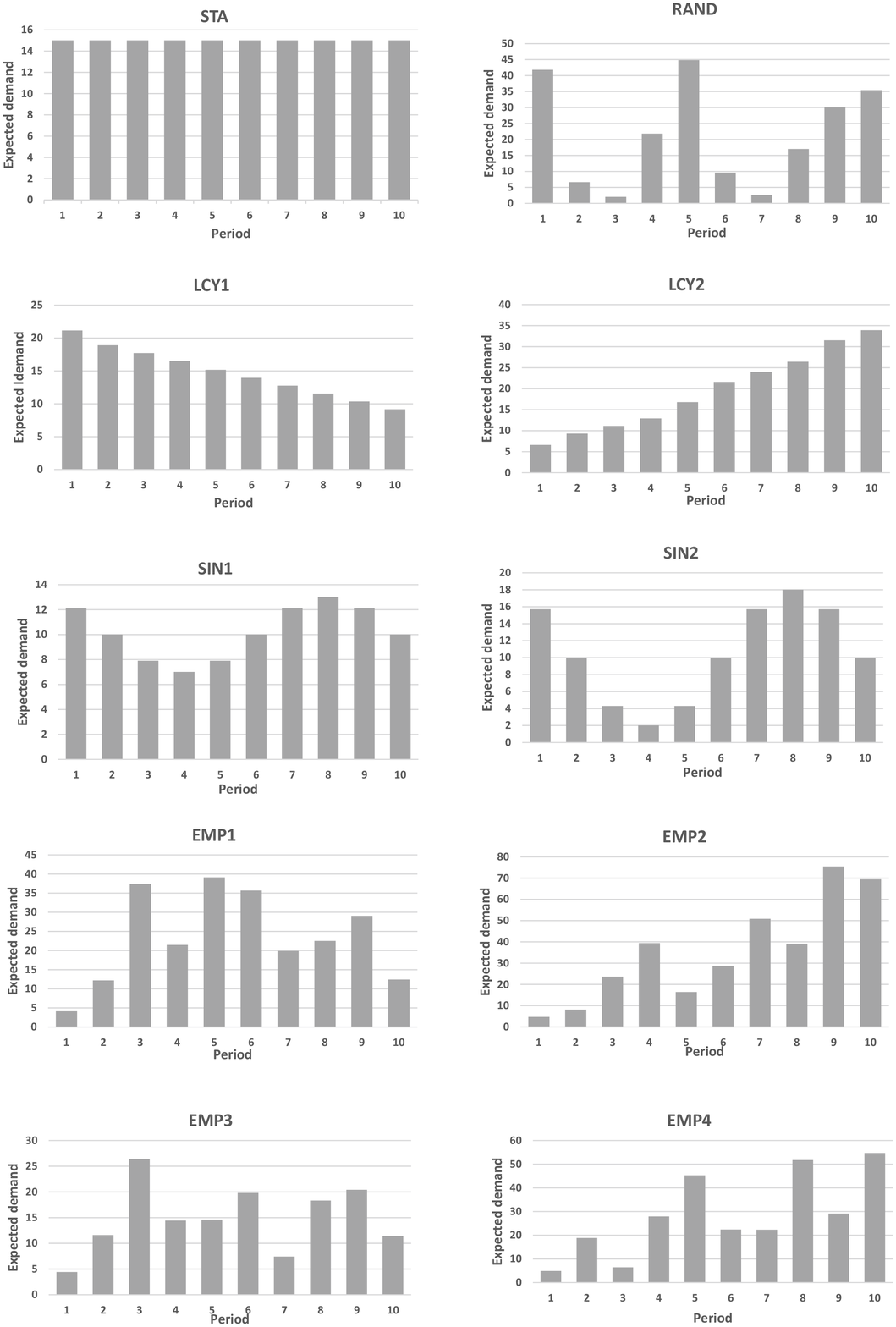}
\caption{Demand patterns in our computational analysis}\label{fig:expecteddemand}
\end{figure}

\end{appendices}

\end{document}

%% file: parameters.tex
\linespread{1.2}
\begin{table}[!ht]
\centering
\caption{Notations adopted in our paper.}
\small
\begin{tabular}{p{4cm}<{\raggedright}p{9cm}<{\raggedright}}
\toprule
Notations  &Description\\
\midrule
Indices  &\\
$n$        &Period index\\
\specialrule{0em}{3pt}{3pt}
Problem parameters  &\\
\specialrule{0em}{1pt}{1pt}
$R_0$  &Initial cash in the first period \\
$x_{0}$  &Initial inventory in the first period\\
$x_n$      &End-of-period inventory in period $n$\\
$R_n$      &End-of-period cash in period $n$\\
$p$      &Selling price\\
$K$      &Fixed ordering cost\\

\multirow{2}*{$W$}      &Overhead costs like rents or wages that the retailer needs to pay in each period\\
$c$      &Unit variable ordering cost\\
$\gamma$      &Unit salvage value for remnant inventory in the final period\\
$D_{u}$      &Maximum possible demand value\\
$D_{b}$      &Minimum possible demand value\\
$B(R)$ & A function of initial cash $R$, representing the cash constraint\\
%$\alpha$  & single-period discount factor, 0\leq\alpha<1 \\
\specialrule{0em}{3pt}{3pt}
\multirow{2}*{Random variables} &\\
\specialrule{0em}{1pt}{1pt}
\multirow{2}*{$\xi_n$}      &Random demand in period $n$, with probability density function $\phi_n(\xi)$, cumulative distribution function $\Phi_n(\xi)$\\
\specialrule{0em}{3pt}{3pt}
Decision variables  &\\
\specialrule{0em}{1pt}{1pt}
$y_{n}$ &Immediate inventory level after ordering in period $n$\\
\bottomrule
\end{tabular}\label{table:notations}
\end{table}
\linespread{1.5}
\normalsize

%% file: example1.tex
\aboverulesep=0ex
\belowrulesep=0ex

\small\noindent
\begin{tabularx}{\linewidth}{@{}@{\extracolsep{\fill}}ll|lllllllllllllll@{}}
  & &\multicolumn{15}{c}{\bm{$R$}}\\
  & &\bf 21&\bf 23&\bf 25&\bf 27&\bf 29&\bf 31&\bf 33 &\bf 35 &\bf 37 &\bf 39 &\bf 41 &\bf 43 &\bf 45 &\bf 47 &\bf 49 \\\hline
\multirow{8}*{\bm{$x$}} &\bf 0
&0 & 0 & 0 & 7 & 9 & 11 & 13 & 14 & 14 & 19 & 21 & 23 & 24 & 24 & 24  \\ 
\cmidrule(lr){6-11}\cmidrule{12-17}
&\bf 1&0 & 0 & 0 & 0 & 7 & 9 & 11 & 13 & 13 & 13 & 19 & 21 & 23 & 23 & 23  \\ 
\cmidrule(lr){7-12}\cmidrule{13-17}
&\bf 2&0 & 0 & 0 & 0 & 0 & 0 & 0 & 0 & 0 & 0 & 19 & 21 & 22 & 22 & 22  \\ 
\cmidrule{13-17}
&\bf 3&0 & 0 & 0 & 0 & 0 & 0 & 0 & 0 & 0 & 0 & 19 & 21 & 21 & 21 & 21  \\ 
\cmidrule{13-17}
&\bf 4&0 & 0 & 0 & 0 & 0 & 0 & 0 & 0 & 0 & 0 & 19 & 20 & 20 & 20 & 20  \\ 
\cmidrule{13-17}
&\bf 5&0 & 0 & 0 & 0 & 0 & 0 & 0 & 0 & 0 & 0 & 19 & 19 & 19 & 19 & 19  \\ 
\cmidrule{13-17}
&\bf 6&0 & 0 & 0 & 0 & 0 & 0 & 0 & 0 & 0 & 0 & 18 & 18 & 18 & 18 & 18 \\ 
\cmidrule{13-17}
&\bf 7&0 & 0 & 0 & 0 & 0 & 0 & 0 & 0 & 0 & 0 & 0 & 0 & 0 & 0 & 0  \\ 
&\bf8 & 0 & 0 & 0 & 0 & 0 & 0 & 0 & 0 & 0 & 0 & 0 & 0 & 0 & 0 & 0  \\ 
\end{tabularx}
\vspace{2pt}

\normalsize

%% file: drawPNy.tex
\begin{figure}[!ht]
\centering
%\pgfplotsset{compat = 1.16} %,height=7.5cm
\begin{tikzpicture}
\begin{axis}[xlabel=$y$,ylabel=$L'(y)$,xmin=0,xmax=100,ymin=0,ymax=300,clip mode=individual,xtick=\empty,ytick=\empty,tick label style={font=\small},
y label style={at={(0,0.5)}},
x label style={at={(0.5,-0.1)}}]

\pgfplotstableread{PNy.txt}\mydata;
\addplot [color=red,line width=1.5pt]
table [x expr=\thisrowno{0}, y expr=\thisrowno{1}] {\mydata};

 \node at (axis cs:45, -10) {\small $S$};
 \addplot[color=blue,dashed,line width=0.5pt] coordinates{ (45, 0) (45, 265)};
 \addplot[color=blue,dashed,line width=0.5pt] coordinates{ (23, 265) (45, 265)};
\node at (axis cs:20, -10) {\small $s$};
 \addplot[color=blue,dashed,line width=0.5pt] coordinates{ (23, 0) (23, 161)};

 \node at (axis cs:11, -10) {\small $x$};
  \addplot[color=blue,dashed,line width=0.5pt] coordinates{ (13, 0) (13, 91)};

  \node at (axis cs:33, -10) {\footnotesize $x+B(C(x))$};
   \addplot[color=blue,dashed,line width=0.5pt] coordinates{ (28, 0) (28, 195)};
   \addplot[color=blue,dashed,line width=0.5pt] coordinates{ (13, 91) (28, 91)};
   \draw[color=blue, decoration={brace,mirror, raise=4pt},decorate,line width=0.5pt]
     (axis cs: 28,91) -- (axis cs:28,195) node[midway,black,right=5pt] {$K$};

 \draw[color=blue, decoration={brace,raise=4pt},decorate,line width=0.5pt]
   (axis cs: 23,161) -- (axis cs:23,265) node[midway,black,left=5pt] {$K$};
\end{axis}
\end{tikzpicture}
\caption{$s$, $C(x)$ and $S$ in the last period $N$}\label{fig: PNy}
\end{figure}

%% file: drawLny.tex
\begin{figure}[!ht]
\centering
\pgfplotsset{height=7cm} % compat = 1.16
\begin{tikzpicture}
\begin{axis}[xlabel=$y$,ylabel=$L_{i\sim n_2}(y)$,xmin=0,xmax=100,ymin=0,ymax=300,clip mode=individual,xtick=\empty,ytick=\empty,tick label style={font=\small},
y label style={at={(0,0.5)}},
x label style={at={(0.5,-0.1)}}]

\pgfplotstableread{PNy.txt}\mydata;
\addplot [color=red,line width=1.5pt]
table [x expr=\thisrowno{0}, y expr=\thisrowno{1}] {\mydata};

 \node at (axis cs:45, -15) {\small $y^\ast$};
 \addplot[color=blue,dashed,line width=0.5pt] coordinates{ (45, 0) (45, 265)};
 \addplot[color=blue,dashed,line width=0.5pt] coordinates{ (21, 255) (39, 255)};
\node at (axis cs:21, -15) {\small $s_i$};
 \addplot[color=blue,dashed,line width=0.5pt] coordinates{ (21, 0) (21, 146)};
  \node at (axis cs:39, -15) {\small $S_i$};
 \addplot[color=blue,dashed,line width=0.5pt] coordinates{ (39, 0) (39, 255)};

 \draw[color=blue, decoration={brace,raise=4pt},decorate,line width=0.5pt]
   (axis cs: 21,146) -- (axis cs:21,255) node[midway,black,left=5pt] {$K$};
\end{axis}
\end{tikzpicture}
\caption{$s_i$, $S_i$, $y^\ast$ and $L_{i\sim n_2}$}\label{fig:Lny}
\end{figure}

%% file: drawLiy.tex
\begin{figure}[!ht]
\centering
\pgfplotsset{height=7 cm}
\begin{tikzpicture}
\begin{axis}[xlabel=$y$,ylabel=$L_i(y)$,xmin=0,xmax=100,ymin=0,ymax=300,clip mode=individual,xtick=\empty,ytick=\empty,tick label style={font=\small},
y label style={at={(0,0.5)}},
x label style={at={(0.5,-0.1)}}]

\pgfplotstableread{PNy.txt}\mydata;
\addplot [color=red,line width=1.5pt]
table [x expr=\thisrowno{0}, y expr=\thisrowno{1}] {\mydata};

 \node at (axis cs:11, -10) {\small $x$};
  \addplot[color=blue,dashed,line width=0.5pt] coordinates{ (13, 0) (13, 91)};

  \node at (axis cs:33, -10) {\footnotesize $x+B(C_i(x))$};
   \addplot[color=blue,dashed,line width=0.5pt] coordinates{ (28, 0) (28, 195)};
   \addplot[color=blue,dashed,line width=0.5pt] coordinates{ (13, 91) (28, 91)};
   \draw[color=blue, decoration={brace,mirror, raise=4pt},decorate,line width=0.5pt]
     (axis cs: 28,91) -- (axis cs:28,195) node[midway,black,right=5pt] {$K$};

\end{axis}
\end{tikzpicture}
\caption{$C_i(x)$ }\label{fig:Liy}
\end{figure}

%% file: computeResults.tex
\linespread{1}
\begin{table}[!ht]
\centering
\caption{Computational study results of different computation methods.}\label{table:NumericalResults}
\begin{tabular}{lccccc}
\toprule
&  \multicolumn{2}{c}{$(s, C(x), S)$ by SDP} & 
\multicolumn{2}{c}{$(s,C(x), S)$ by MIP} &\multirow{2}*{Cases}\\
\cmidrule(lr){2-3} \cmidrule(lr){4-5}
 & MAX &AVG  & MAX & AVG  \\
\midrule
Fixed ordering cost \\
\specialrule{0em}{1pt}{1pt}
10    &0.41\% &0.00\%  &4.85\%    &1.35\%   &90\\
15   &0.27\% &0.00\% &10.91\% &1.93\%  &90\\
20   &0.65\% &0.03\%  &17.66\%    &2.87\%   &90\\
\specialrule{0em}{2pt}{2pt}
Margin \\
\specialrule{0em}{1pt}{1pt}
4   &0.65\% &0.03\%  &17.66\%    &2.87\%   &90\\
5  &0.27\% &0.00\%  &10.55\%    &1.89\%   &90\\
6 &0.02\% &0.00\%  &5.49\%    &1.39\%   &90\\
\specialrule{0em}{2pt}{2pt}
Cash capacity\\
\specialrule{0em}{1pt}{1pt}
6  &0.65\% &0.02\% &17.66\%    &2.47\%   &90\\
8   &0.45\% &0.01\%  &7.85\%    &1.99\%   &90 \\
10   &0.33\% &0.00\%  &8.17\%    &1.70\%   &90 \\
\specialrule{0em}{2pt}{2pt}
Demand pattern\\
\specialrule{0em}{1pt}{1pt}
STA   &0.33\% &0.01\%  &2.63\%    &0.91\%   &27\\
LC1   & 0.00\%  &0.00\%  & 2.50\%   &1.14\%   &27\\
LC2  &0.04\%  &0.01\%  & 3.07\%   &1.68\%   &27\\
SIN1    &0.65\%  &0.06\%  &4.93\%    &1.35\%   &27\\
SIN2    &0.41\%  &0.04\%  & 17.66\%    & 2.93\% &27\\
RAND   &0.00\% &0.00\%  &3.17\%    & 1.20\%  &27\\
EMP1    &0.01\% &0.00\% &10.04\%   & 3.09\%   &27\\
EMP2    &0.00\% &0.00\% & 6.28\%   & 2.26\%  &27\\
EMP3   &0.03\% &0.00\% & 11.00\%   & 3.44\%  &27\\
EMP4   &0.01\% &0.00\%  & 9.09\%   & 2.52\%  &27\\
\specialrule{0em}{3pt}{3pt}
Gap over 1\%  &\multicolumn{2}{c}{0}&\multicolumn{2}{c}{182}\\
Avg Time  &\multicolumn{2}{c}{305s}&\multicolumn{2}{c}{0 s}\\
General  &0.65\%~~~ &0.01\%  &17.66\%    & 2.05\%  &270\\
\bottomrule
\end{tabular}
\end{table}

%% file: demandValues.tex
\begin{table}[!ht]
\centering
\caption{Expected demand values for different demand patterns.}\label{table:demand datas}
\begin{tabular}{lrrrrrrrrrr}
\toprule
Demand pattern & \multicolumn{10}{c}{Expected demand values}\\
\midrule
STA& 15& 15& 15& 15& 15& 15& 15& 15 & 15 & 15\\
LC1& 21.15& 18.9& 17.7& 16.5& 15.15& 13.95& 12.75& 11.55 &10.35 & 9.15\\
LC2& 6.6& 9.3& 11.1& 12.9& 16.8& 21.6& 24& 26.4 & 31.5 & 33.9\\
SIN1& 12.1 & 10 & 7.9 & 7 & 7.9 & 10 & 12.1 & 13 & 12.1 & 10\\
SIN2& 15.7 & 10 & 4.3 & 2 & 4.3 & 10 & 15.7 & 18 & 15.7 & 10\\
RAND& 41.8& 6.6& 2& 21.8& 44.8& 9.6& 2.6& 17 &30 & 35.4\\
EMP1& 4.08& 12.16& 37.36& 21.44& 39.12& 35.68& 19.84& 22.48 &29.04 & 12.4\\
EMP2& 4.7& 8.1& 23.6& 39.4& 16.4& 28.7& 50.8& 39.1 &75.4 & 69.4\\
EMP3& 4.4& 11.6& 26.4& 14.4& 14.6& 19.8& 7.4& 18.3 &20.4 & 11.4\\
EMP4& 4.9& 18.8& 6.4& 27.9& 45.3& 22.4& 22.3& 51.7 &29.1 & 54.7 \\
\bottomrule
\end{tabular}
\end{table}